\documentclass[a4paper,11pt]{amsart}

\usepackage{amsmath}
\usepackage{amsfonts}
\usepackage{amssymb}
\usepackage{graphicx}
\usepackage[abbrev,alphabetic]{amsrefs}
\usepackage{color}
\usepackage{soul,xcolor} 
\setstcolor{red}

\usepackage{amsthm}
\usepackage{comment}
\usepackage[all,cmtip]{xy}

\newcommand{\Fp}{\mathbb{\overline{F}}_p}

\newcommand{\mbQ}{\mathbb{Q}}

\newcommand{\mbR}{\mathbb{R}}

\newcommand{\mbZ}{\mathbb{Z}}
\newcommand{\mbP}{\mathbb{P}}

\newcommand{\mcO}{\mathcal{O}}

\DeclareMathOperator{\Supp}{Supp}

\DeclareMathOperator{\Pic}{Pic}

\newcommand*{\coloneq}{\mathrel{\mathop:}=}

\theoremstyle{plain}
\newtheorem{theorem}{Theorem}[section]
\newtheorem{proposition}[theorem]{Proposition}
\newtheorem{lemma}[theorem]{Lemma}

\newtheorem{question}[theorem]{Question}

\theoremstyle{definition}
\newtheorem{definition}[theorem]{Definition}

\theoremstyle{remark}
\newtheorem{remark}[theorem]{Remark}

\title[On base point free theorem for lc threefolds over $\Fp$]
{On base point free theorem and Mori dream spaces for log canonical threefolds over 
the algebraic closure of a finite field
}

\author{Yusuke Nakamura}
\address{Graduate School of Mathematical Sciences, 
The University of Tokyo, 3-8-1 Komaba, Meguro-ku, Tokyo 153-8914, Japan.}
\email{nakamura@ms.u-tokyo.ac.jp}

\author{Jakub Witaszek}
\address{Department of Mathematics, Imperial College, London, 180 Queen's Gate, 
London SW7 2AZ, UK} 
\email{j.witaszek14@imperial.ac.uk}

\begin{document}
\begin{abstract}
The authors and D.\ Martinelli proved in \cite{MNW} the base point free theorem for big line bundles 
on a three-dimensional log canonical 
projective pair defined over the algebraic closure of a finite field. 
In this paper, we drop the bigness condition when the characteristic is larger than five. 
Additionally, we discuss Mori dream spaces defined over the algebraic closure of a finite field. 
\end{abstract}


\maketitle

\section{Introduction}
One of the most prominent problems in algebraic geometry is to classify all smooth projective varieties. In characteristic zero, this goal has been partially achieved due to the developments of the Minimal Model Program.

	In order to construct a single step of the MMP, it is necessary to invoke the base point free theorem (cf.\ \cite[Theorem 3.1.1]{KMM}), which asserts that a nef line bundle $D$ 
on a Kawamata log terminal projective pair $(X, \Delta)$ 
defined over an algebraically closed field of characteristic zero 
is semiample when $D - (K_X + \Delta)$ is nef and big. Unfortunately, the existence of a resolution of singularities 
and the Kawamata--Viehweg vanishing theorem are essential in the proof, and hence it is not valid in positive characteristic. Even so, Tanaka managed to prove the base point free theorem for positive characteristic surfaces, and Keel showed it for threefolds over the algebraic closure of a finite field $\Fp$ in the case when $D$ is big.

	Recent years brought further breakthrough results. The progress in the positive characteristic Minimal Model Program (\cite{HX}, \cite{CTX}, \cite{Birkar}), led to a proof of the base point free theorem for threefolds over algebraically closed fields of characteristic $p$ larger than five, with an assumption that $D$ is big (\cite{Birkar} and \cite{Xu}), and, eventually, without the bigness assumption (\cite{BW}, see Theorem \ref{theorem:bpfBW}).

Over the algebraic closure $\Fp$ of a finite field,  
the authors and D.\ Martinelli proved in \cite{MNW} a stronger result in the case when $D$ is big. 

\begin{theorem}[{\cite[Theorem 1.1]{MNW}}]\label{theorem:MNW}
Let $(X, \Delta)$ be a three-dimensional projective log canonical pair defined over $\Fp$, 
and let $D$ be a nef and big $\mbQ$-Cartier $\mbQ$-divisor on $X$. 
If $D - (K_X + \Delta)$ is also nef and big, then $D$ is semiample. 
\end{theorem}

Furthermore, in \cite{MNW} the above assertion is proved also for three-dimensional projective 
(possibly non-log canonical) log pairs $(X, \Delta)$ such that all the coefficients of $\Delta$ are at most one and each irreducible component of 
$\Supp (\lfloor \Delta \rfloor)$ is normal. This is a generalization of Keel's result \cite{Keel}, who showed the same assertion under the stronger assumption that $\lfloor \Delta \rfloor = 0$. 

We also note that Theorem \ref{theorem:MNW} is not true for any other algebraically closed field
$k \not = \Fp$ even in the two-dimensional case (see \cite[Example A.3]{Tanaka}).

The main aim of this paper is to drop the assumption that $D$ is big in Theorem \ref{theorem:MNW},
when the characteristic is larger than five: 
\begin{theorem}[Main theorem]\label{theorem:main}
Let $(X, \Delta)$ be a three-dimensional projective log canonical pair defined over $\Fp$ 
of characteristic $p$ larger than five, 
and let $D$ be a nef $\mbQ$-Cartier $\mbQ$-divisor on $X$. 
If $D - (K_X + \Delta)$ is nef and big, then $D$ is semiample. 
\end{theorem}

\noindent
In the proof of Theorem \ref{theorem:main}, 
we run a $D$-crepant MMP and finish off by combining three types of base point free theorems: 
Theorem \ref{theorem:MNW}, Theorem \ref{theorem:MNWsurf}, and Theorem \ref{theorem:bpfBW}. 
The key ingredient for running this MMP is the minimal model program developed in \cite{Birkar} and \cite{BW}, 
and the boundedness of the lengths of extremal rays (\cite[Theorem 1.1 (ii)]{BW}). 
In order to use results from \cite{BW}, we need to assume that $p > 5$.

The ultimate goal of the research conducted here and in \cite{MNW} 
is to gain a better understanding of new phenomena happening over $\Fp$, 
and seek further results around the minimal model program which are valid only over these algebraically closed fields. 
In the spirit of this, it is natural to ask the following question.
\begin{question} \label{question:main}  Let $(X,\Delta)$ be a log canonical projective three-dimensional pair defined over $\Fp$. Assume that $-(K_X+\Delta)$ is nef and big. Is $X$ a Mori dream space?
\end{question} 

When $(X,\Delta)$ is klt and the base field is of characteristic zero, 
then it is well known that the answer to the above question is positive (\cite[Corollary 1.3.2]{BCHM}), 
but without the klt assumption this does not hold.

The aforementioned question is motivated by the validity of the log canonical base point free theorem in characteristic $p>5$ (Theorem \ref{theorem:main}). Under the assumption of Question \ref{question:main}, if $p>5$, then any nef line bundle on $X$ is semiample. Indeed, if $L$ is a nef line bundle, then
\[
L - (K_X + \Delta)
\]
is nef and big, and so $L$ is semiample by Theorem \ref{theorem:main}. 
In particular, $X$ satisfies one of the conditions in the definition of a Mori dream space 
(cf.\ Definition \ref{definition:mds}).

Unfortunately, the answer to the above question is in general negative, as we show in Theorem \ref{theorem:main_mds}. 

\medskip

We note that some results on the Minimal Model Program for log canonical threefolds defined over algebraically closed fields of characteristic larger 
than five have been recently established in \cite{Waldron}. Further, the abundance conjecture for log canonical fourfolds in characteristic zero with a big boundary has been recently obtained in \cite{H}.

\pagebreak
\subsection*{Notation and conventions}
\begin{itemize}
\item When we work over a normal variety $X$, we often identify a line bundle $L$ 
with the divisor corresponding to $L$. 
For example, we use the additive notation $L + A$ for a line bundle $L$ and a divisor $A$. 

\item Following the notation of \cite{Keel}, for a morphism $f \colon X \to Y$ and a line bundle $L$ on $Y$, 
we denote by $L |_X$ the pullback $f^* L$. 

\item We say that a line bundle $L$ on $X$ is \textit{semiample} 
when the linear system $|mL|$ is base point free for some integer $m$. 
When $L$ is semiample, the surjective map $f \colon X \to Y$, defined by $|mL|$ 
for a sufficiently large and divisible positive integer $m$, satisfies 
$f_* \mcO _X = \mcO _Y$. 
We call $f$ \textit{the semiample fibration associated to $L$}. 

\item A \textit{log pair} $(X, \Delta)$ is a normal variety $X$ 
together with an effective $\mbQ$-divisor $\Delta$ such that 
$K_X + \Delta$ is $\mbQ$-Cartier. 

For a proper birational morphism $f \colon X' \to X$ from a normal variety $X'$, 
we write 
\[
K_{X'} + \sum _i a_i E_i = f^* (K_X + \Delta), 
\]
where $E_i$ are prime divisors. 
We say that the pair $(X, \Delta)$ is \textit{log canonical} 
if $a_i \le 1$ for any proper birational morphism $f$. 
Further, we say that the pair $(X, \Delta)$ is \textit{Kawamata log terminal} (\textit{klt} for short)
if $a_i < 1$ for any proper birational morphism $f$. 

We call a log pair $(X, \Delta)$ 
\textit{divisorial log terminal} (\textit{dlt} for short) 
when there exists a log resolution $f \colon X' \to X$ such that 
$a_i < 1$ for any $f$-exceptional divisor $E_i$ on $X'$.
\item We denote by $\mathrm{Nef}(X)$ the cone of nef $\mbR$-divisors on $X$ up to numerical equivalence 
and by $\mathrm{Mov}(X)$ the cone of mobile $\mbR$-divisors on $X$ up to numerical equivalence. 
\item We denote by $\Pic^0(X)$ the identity component of the Picard scheme 
$\Pic(X)$ of a variety $X$ defined over a field $k$. 
We call $\mathrm{NS}(X) \coloneq \Pic(X)/\Pic^0(X)$ the Neron-Severi group. 
\end{itemize}

\section{Preliminaries}
In this section, we list some propositions about semiampleness of line bundles, which are necessary for the proof of 
Theorem \ref{theorem:main}.

First, on projective curves over $\Fp$, the nefness and the semiampleness are equivalent. 
\begin{lemma}\label{lemma:curve}
Let $D$ be a nef line bundle on a projective curve defined over $\Fp$. 
Then $D$ is semiample. 
\end{lemma}
\begin{proof}
This easily follows from the fact that a numerically trivial line bundle on 
a projective variety over $\Fp$ is always torsion (cf.\ \cite[Lemma 2.16]{Keel}). 
\end{proof}

On projective surfaces over $\Fp$, a nef and big line bundle is always semiample. 
\begin{theorem}[{\cite[Theorem 2.9]{Artin}}, cf.\ {\cite[Corollary 0.3]{Keel}}]\label{theorem:Artin}
Let $D$ be a nef and big line bundle on a projective surface defined over $\Fp$. 
Then $D$ is semiample. 
\end{theorem}

Further, on  projective surfaces over $\Fp$, the base point free theorem is known 
under much weaker assumptions than usual. 

\begin{theorem}[{\cite[Theorem 1.4]{MNW}}]\label{theorem:MNWsurf}
Let $D$ be a nef $\mbQ$-divisor on a normal projective surface $S$ defined over $\Fp$, 
and let $\Delta$ be an effective $\mbQ$-divisor on $S$. 
If $D - (K_S + \Delta)$ is also nef, then $D$ is semiample. 
\end{theorem}

The semiampleness is preserved under taking the pullback under a surjective morphism to a normal variety. 
\begin{lemma}[{cf.\ \cite[Lemma 2.10]{Keel}}]
\label{lemma:surj}
Let $f \colon Y \rightarrow X$ be a proper surjection between varieties defined over 
an algebraically closed field and let $L$ be a line bundle on $X$. 
Assume that $X$ is normal. Then $L$ is semiample if and only if $f^* L$ is semiample.
\end{lemma}

The following  technical proposition, shown in \cite{BW}, 
is used in the proof of Lemma \ref{lemma:crepant} (see Remark \ref{remark}). 
\begin{proposition}[{\cite[Lemma 3.6]{BW}}]\label{proposition:lift}
Let $(X, \Delta)$ be a three-dimensional klt pair defined over an algebraically closed field 
of characteristic larger than five. 
Let $C$ be a curve on $X$ and let $W \to X$ be a log resolution of $(X, \Delta)$. 
Then there exists a curve $C'$ on $W$ such that the induced map $C' \to X$ 
is birational onto $C$. 
\end{proposition}

Lastly, for convenience of the reader we state the base point free theorem established by Birkar and Waldron.
\begin{theorem}[{\cite[Theorem 1.2]{BW}}] \label{theorem:bpfBW} 
Let $(X,\Delta)$ be a three-dimensional projective klt pair defined over an algebraically closed field $k$ 
of characteristic larger than five. 
Assume that $D$ is a $\mbQ$-Cartier $\mbQ$-divisor such that $D$ is nef and 
$D-(K_X + \Delta)$ is nef and big. Then $D$ is semiample.
\end{theorem}

\section{Proof of the main theorem}

We start by showing an auxiliary lemma. For technical reasons, instead of simply assuming that $D$ is Cartier, 
we only assume that it is $\mbQ$-linearly equivalent to a Cartier divisor after taking the pullback under a birational morphism (see Remark \ref{remark}).

\begin{lemma} \label{lemma:crepant} 
Let $(X,\Theta)$ be a $\mbQ$-factorial three-dimensional projective klt pair 
defined over an algebraically closed field of characteristic larger than five. 
Let $D$ be a nef $\mbQ$-divisor, and let $S$ be an effective divisor. 
Fix $\epsilon \in \mbQ_{>0}$ such that $\epsilon \le \frac{1}{7}$.  
Further, assume that 
\begin{enumerate}
 \item $D - \epsilon S \sim_{\mbQ} K_X + \Theta$,
 \item $(1-\epsilon)S \le \Theta$, and
 \item there exists a birational morphism $\pi \colon X' \to X$ such that 
 $\pi^*D$ is $\mbQ$-linearly equivalent to a Cartier divisor. 
\end{enumerate}
Let $g \colon X \to V$ be a contraction of a $(D-\epsilon S)$-extremal ray $R$. 
Then there exists a $\mbQ$-Cartier $\mbQ$-divisor $D'$ on $V$ such that $D \sim_{\mbQ} g^*D'$.
\end{lemma}
In particular, if $g$ is birational, then $D \sim_{\mbQ} g^* g_* D$. 
The main example of a situation when the assumptions of the lemma hold, 
is when for a dlt pair $(X,\Delta)$ we have: $S = \lfloor \Delta \rfloor$, 
$\Theta = \Delta - \epsilon S$, and $D = K_X + \Delta$ is Cartier. 
\begin{proof}[{Proof of Lemma \ref{lemma:crepant}}]
By a standard argument, it is enough to show that $D \equiv_{g} 0$, 
that is, $D \cdot R = 0$ (cf.\ \cite[Theorem 3.7 (4)]{KollarMori}). 
To this end, we note that the nefness of $D$ implies that $R$ is also a $(D-S)$-negative extremal ray. 
For $\Theta' \coloneq \Theta - (1-\epsilon)S$, we have that
\[
D - S \sim _{\mbQ} K_{X} + \Theta',
\]
and $(X,\Theta')$ is klt.

By the boundedness of the lengths of extremal rays of $K_X + \Theta'$ (\cite[Theorem 1.1 (ii)]{BW}), 
there exists a curve $C$ on $X$ which generates 
the $(K_{X} + \Theta ')$-negative extremal ray $R$ and satisfies 
\[
(D - S) \cdot C = (K_{X} + \Theta ') \cdot C \ge -6. 
\]
Therefore, we obtain
\begin{align*}
0 > (D - \epsilon S) \cdot C &= \epsilon (D - S) \cdot C + (1 - \epsilon)D \cdot C \\
&\ge -6 \epsilon + (1 - \epsilon)D \cdot C, 
\end{align*}
and hence $D \cdot C < \frac{6 \epsilon}{1-\epsilon} \leq 1$. 

Without loss of generality we can assume that 
$\pi \colon X' \to X$ is a log resolution of singularities of $(X,\Delta)$. 
By Proposition \ref{proposition:lift}, 
there exists a curve $C'$ on $X'$ such that 
the restriction $\pi|_{C'}$ is a birational map onto $C$. 
By the projection formula, we obtain
\[
D \cdot C = \pi^* D \cdot C', 
\]
and since $\pi^*D$ is $\mbQ$-linearly equivalent to a Cartier divisor, we have $D \cdot C \in \mbZ$. 
But $D \cdot C < 1$, and so the nefness of $D$ implies that $D \cdot C = 0$, 
which completes the proof. 
\end{proof}

Now, we can proceed with the proof of the main theorem.
\begin{proof}[Proof of Theorem \ref{theorem:main}]
First, we can assume that $D$ is Cartier. Indeed, $mD$ is nef and 
\[
mD - (K_X + \Delta) = (m-1)D + D -(K_X+\Delta)
\]
is nef and big for any natural number $m$, and so without loss of generality we can replace $D$ by $mD$. 
Furthermore, we may assume that $(X, \Delta)$ is $\mbQ$-factorial and dlt by applying a dlt-modification 
(\cite[Theorem 1.6]{Birkar}) and Lemma \ref{lemma:surj}. Set 
\[
L := D - (K_X + \Delta) \text{, and} \quad S := \lfloor \Delta \rfloor. 
\]

For a fixed $0 < \epsilon \le \frac{1}{7}$, we consider a $(D - \epsilon S)$-MMP. 
We note that 
\[
D - \epsilon S \sim _{\mbQ} (K_X + \Delta - \epsilon S) + L. 
\]
Since $(X, \Delta - \epsilon S)$ is klt, and $L$ is nef and big, 
there exists a $\mbQ$-divisor $\Theta$ such that
\begin{itemize}
\item  $(X, \Theta)$ is klt, 
\item $D - \epsilon S \sim _{\mbQ} K_X +  \Theta$, and 
\item $(1- \epsilon)S \le \Theta$.
\end{itemize} 
Hence, we may run a $(D - \epsilon S)$-MMP with scaling of some ample divisor 
and it terminates (\cite[Theorem 1.6]{BW}): 
\[
X=X_0 \overset{f_0}{\dasharrow} X_1 \overset{f_1}{\dasharrow} \cdots \overset{f_{l-1}}{\dasharrow} X_l =: Y. 
\]
Note that this is also a $(K_X + \Theta)$-MMP.

We set $D_i$, $S_i$, and $\Theta_i$ to be the strict transforms of $D$, $S$, and $\Theta$ on $X_i$, respectively. 
Then each $f_i$ is either a 
$(D_i - \epsilon S_i)$-negative divisorial contraction or a $(D_i - \epsilon S_i)$-flip. 

By subsequently applying Lemma \ref{lemma:crepant} to $X_i, \Theta_i, D_i, S_i$, for $0 \leq i < l$, 
and the corresponding contractions, we get that this $(K_X + \Theta)$-MMP is $D$-crepant, 
that is, there exists a resolution $W$ of the indeterminacy of $f \coloneq f_{l-1} \circ \cdots \circ f_0$: 
\[\xymatrix{
& W \ar[dl] _{p} \ar[dr] ^{q} & \\
X \ar@{-->}[rr] ^{f} & &Y
}\]
such that $p^*D \sim_{\mbQ} q^* f_* D$ holds. 
To this end, it is sufficient to show inductively for each $i$ that 
the conditions (1), (2), (3) of Lemma \ref{lemma:crepant} are satisfied and that $D_i$ is nef.
The first two hold since the containment of divisors and 
the $\mbQ$-linear equivalence are preserved under taking the pushforward under a divisorial contraction or a flip. 
Thus, we are left to show (3) and the nefness of $D_i$. 
Since these hold clearly for $i=0$, we can assume that $i>0$. 
Further, by induction, we can assume that $f'_{i-1} \coloneq f_{i-1} \circ \cdots \circ f_0$ is $D$-crepant. 
In particular, there exists a commutative diagram
\[\xymatrix{
& W_i \ar[dl] _{p_i} \ar[dr] ^{q_i} & \\
X \ar@{-->}[rr] ^{f'_{i-1}} & &X_i. 
}\]
such that $q_i^* D_i \sim_{\mbQ} p_i^* D$, where $p_i^*D$ is Cartier. Hence, (3) is satisfied, and $D_i$ is nef.

Let $f\colon X \dasharrow Y$ be the end result of this MMP and let $S'$ and $D'$ be the strict transforms of 
$S$ and $D$ on $Y$, respectively. 
Then they satisfy one of the following properties:
\begin{itemize}
\item[(A)] $D ' - \epsilon S'$ is nef, or
\item[(B)] there exists a $(D ' - \epsilon S')$-negative extremal contraction 
$g\colon Y \to Z$ with $\dim Z \le 2$. 
\end{itemize}

\noindent
Since $f$ is $D$-crepant, $D'$ is nef. 
Further, in order to prove the semiampleness of $D$, 
it is sufficient to show the semiampleness of $D'$ (cf.\ Lemma \ref{lemma:surj}). 

We divide the argument into the cases (A) and (B). 

\vspace{2mm}
\noindent
\underline{Case (A).}\ \ 
First, we prove that $D' - \epsilon S'$ is semiample. 
Let $\Theta'$ be the strict transform of $\Theta$ on $Y$. 
Then $(Y, \Theta ')$ is klt and 
\[
D ' - \epsilon S' \sim _{\mbQ} K_{Y} + \Theta '.
\]
Further, $\Theta '$ is big, as $\Theta$ is big, and thus we can write $\Theta ' = \Gamma + A$ 
with an effective $\mbQ$-divisor $\Gamma$ and an 
ample $\mbQ$-divisor $A$, so that $(Y, \Gamma)$ is klt. 
Since 
\[
D ' - \epsilon S' - (K_{Y} + \Gamma) \sim _{\mbQ} A
\]
is ample, 
the nef  $\mbQ$-divisor $D' - \epsilon S'$ turns out to be semiample 
by the base point free theorem for klt pairs (Theorem \ref{theorem:bpfBW}). 

Moreover,  $D' - (\epsilon - \delta) S'$ is semiample for a sufficiently small $\delta > 0$. 
Indeed, 
\[
D ' - (\epsilon - \delta) S' - (K_{Y} + \Gamma) \sim _{\mbQ} A + \delta S'
\]
is ample, when $\delta$ is small enough. 
Further, $D' - (\epsilon - \delta) S'$ is nef, as $D'$ and $D' - \epsilon S'$ are nef. 
Hence, the semiampleness of $D' - (\epsilon - \delta) S'$ follows from the base point free theorem (Theorem \ref{theorem:bpfBW}).
 
Let $h \colon Y \to V$ be the semiample fibration associated to $D' - (\epsilon - \delta) S'$ and 
let $H$ be an ample $\mbQ$-Cartier $\mbQ$-divisor on $V$ such that 
\[
D' - (\epsilon - \delta) S' \sim _{\mbQ} h^* H. 
\]

Let $C$ be a curve contracted by $h$. In particular, $\big( D' - (\epsilon - \delta) S' \big) \cdot C = 0$. 
Since $D'$ and $D' - \epsilon S'$ are nef, we get that $(D' - \epsilon S') \cdot C = 0$, and so
the semiample fibration associated to $D' - \epsilon S'$ factors through $h$. 
Hence, $D - \epsilon S'$ can be written as a pullback under $h$ of 
a $\mbQ$-Cartier $\mbQ$-divisor on $V$, which implies that $S'$ can be written as 
a pullback of an effective $\mbQ$-Cartier $\mbQ$-divisor.

Therefore, we may write 
\[
D' \sim _{\mbQ} h^* (H + E)
\]
with an ample $\mbQ$-Cartier $\mbQ$-divisor $H$ and an effective $\mbQ$-Cartier $\mbQ$-divisor $E$. 

If $\dim V \le 2$, then the nef and big $\mbQ$-divisor $H+E$ is semiample by Lemma \ref{lemma:curve} and 
Theorem \ref{theorem:Artin}. 
Hence, $D'$ is semiample, and so is $D$. 

If $\dim V = 3$, then $D'$ is big. Therefore, $D$ is also big, since $f$ is $D$-crepant. 
In this case, the assertion is nothing but Theorem \ref{theorem:MNW}. 

This completes the proof in the case (A). 

\vspace{2mm}
\noindent
\underline{Case (B).}\ \ 
First,  by Lemma \ref{lemma:crepant}, 
there exists a $\mbQ$-Cartier $\mbQ$-divisor $D_Z$ on $Z$ such that $D' \sim _{\mbQ} g^* D_Z$. 
In particular, $g$ is $S'$-positive. 

If $\dim Z \le 1$, the nef $\mbQ$-divisor $D_Z$ is semiample (Lemma \ref{lemma:curve}). 
Hence, $D'$ is semiample, and so is $D$. 

In what follows, we assume that $\dim Z = 2$. 
Since $g$ is $S'$-positive, some irreducible component $S' _1$ of $S'$ dominates $Z$. 
Let $S_1$ and $\overline{S}_1$ be the strict transforms of $S_1 '$ on $X$ and $W,$ respectively. 
\[\xymatrix{
\overline{S}_1 \ar@{}|{\rotatebox{0}{$\subset$}}[r] & W \ar[d] _{p} \ar[dr] ^{q} & & \\
S_1 \ar@{}|{\rotatebox{0}{$\subset$}}[r] & X \ar@{-->}[r] _f & Y \ar[r] _g  & Z. \\
&& S_1' \ar@{}|{\rotatebox{90}{$\subset$}}[u] \ar@{->>}[ur] &
}\]
Note that $S_1$ is normal, since $(X, \Delta)$ is dlt 
(\cite[Proposition 4.1]{HX}).
Hence, by Lemma \ref{lemma:surj}, in order to prove the semiampleness of $D_Z$, 
it is sufficient to show the semiampleness of $D_Z |_{\overline{S}_1}$. 
Since $D_Z |_{\overline{S}_1} \sim_{\mbQ} D |_{\overline{S}_1}$, 
it is sufficient to show the semiampleness of $D |_{S_1}$. 
By adjunction, the new pair $(S_1, \Delta _{S_1})$ defined by 
$K_{S_1} + \Delta _{S_1} = (K_X + \Delta)| _{S_1}$ is also dlt. 
Note that $D |_{S_1}$ and 
\[
D |_{S_1} - (K_{S_1} + \Delta _{S_1}) \sim _{\mbQ} \big( D - (K_X + \Delta) \big) |_{S_1}
\]
are nef. Therefore, the semiampleness of $D |_{S_1}$ follows from Theorem \ref{theorem:MNWsurf}. 
This shows that $D_Z$ is semiample, and so is $D$. 

The proof in the case (B) is completed. 
\end{proof}

\begin{remark}\label{remark}
The same argument as in Lemma \ref{lemma:crepant} appears in \cite{Kawamata}. 
One difference from \cite{Kawamata} is that we do not know whether $D_i$ is Cartier in our case. 
If we worked in characteristic zero, we could prove that $D_i$ is Cartier by the base point free theorem. 
The point is that the original base point free theorem in characteristic zero is stronger: 
for a nef line bundle $D$ on a klt pair $(X, \Delta)$ satisfying that $D - (K_X + \Delta)$ is nef and big, 
it is proved that $mD$ is base point free for sufficiently large $m$. 
In positive characteristic, we only know the base point freeness for a
sufficiently divisible $m$ (Theorem \ref{theorem:MNW} and Theorem \ref{theorem:bpfBW}). 
Proposition \ref{proposition:lift} was necessary to overcome this issue. 
\end{remark}

\section{Mori dream spaces over $\Fp$} \label{section:mds}
The main goal of this section is to show Theorem \ref{theorem:main_mds}, 
which gives a negative answer to Question \ref{question:main}. 

First, we recall the definition of a Mori dream space introduced by Hu and Keel (\cite{HK}). 
\begin{definition} \label{definition:mds} 
A normal projective variety $X$ is called a \emph{Mori dream space} 
if the following three conditions hold:
\begin{enumerate}
	\item $X$ is $\mbQ$-factorial and $\Pic(X)_{\mbQ} = \mathrm{NS}(X)_{\mbQ}$,
	\item $\mathrm{Nef}(X)$ is an affine hull of finitely many semiample line bundles, and
	\item there is a finite collection of small $\mbQ$-factorial modifications $f_i \colon X \dashrightarrow X_i$ such that each $X_i$ satisfies (2) and $\mathrm{Mov}(X)$ is the union of $f^*_i(\mathrm{Nef}(X_i))$.
\end{enumerate}
\end{definition}
\noindent
In this definition, a \emph{small $\mbQ$-factorial modification} of a projective variety $X$ 
is a birational contraction $f \colon X \to X'$, with $X'$ projective and $\mbQ$-factorial, such that $f$ is an isomorphism in codimension one.

\begin{remark} Note that there are many examples of Mori dream spaces which appear only over $\Fp$. For example, an elliptic curve $E$ defined over an algebraically closed field $k$ is a Mori dream space if and only if $k = \Fp$ for some prime number $p$.
\end{remark}

\subsection{Two-dimensional case}
Before proceeding with the proof of Theorem \ref{theorem:main_mds}, 
we discuss the situation in the two-dimensional case.

Assume that $(X,\Delta)$ is a two-dimensional projective log canonical pair defined over $\Fp$. 
When $-(K_X+\Delta)$ is nef and big, then $X$ is a Mori dream space by the following proposition. 

\begin{proposition} Let $X$ be a normal projective surface defined over $\Fp$ such that $-K_X$ is big. Then $X$ is a Mori dream space. 
\end{proposition}
\begin{proof}
First, by \cite[Theorem 4.5]{Tanaka}, a normal projective surface defined over $\Fp$ is always $\mbQ$-factorial. 
Further, a numerically trivial line bundle on a projective variety defined over $\Fp$ is always torsion  (cf.\ \cite[Lemma 2.16]{Keel}), 
and so $\Pic(X)_{\mbQ} = \mathrm{NS}(X)_{\mbQ}$. Thus, by definition, it is enough to verify that the nef cone of $X$ is a rational polyhedron and any nef line bundle is semiample. 

In order to prove the first assertion, write $-K_X \sim_{\mbQ} A + E$ for some ample $\mbQ$-divisor $A$ and a $\mbQ$-effective $\mbQ$-divisor $E$. Then, the cone of effective curves is a rational polyhedron by Tanaka's cone theorem (\cite[Theorem 3.13]{Tanaka}) applied to $K_X + E$, and so its dual cone, the nef cone $\mathrm{Nef}(X)$, is a rational polyhedron as well.   

As for the second assertion, take any nef divisor $D$ on $X$. Then 
\[ D - (K_X+E) \sim_{\mbQ} A+D
\]
is ample, and so $D$ is semiample by Theorem \ref{theorem:MNWsurf} (see also \cite[Theorem 4.15]{Tanaka}). This concludes the proof. 
\end{proof}

\begin{remark}
For any algebraically closed field $k$, it is known that 
a smooth projective rational surface with big anticanonical divisor $-K_X$ is a Mori dream space (\cite{TVV}). 
\end{remark}

Note that the answer to Question \ref{question:main} is negative even in dimension two, when $X$ is defined over an algebraically closed field different from $\Fp$.

\begin{proposition} 
For any algebraically closed field $k$ different from $\Fp$, 
there exists a log canonical projective two-dimensional pair $(X,\Delta)$ defined over $k$, 
such that $-(K_X+\Delta)$ is nef and big, but not every nef line bundle on $X$ is semiample. In particular, $X$ is not a Mori dream space.
\end{proposition}
\begin{proof}
The construction is based on Cutkosky's example (see \cite[Section 2.3.B]{Lazarsfeld1}). 
Take an elliptic curve $C$, a point $P \in C$, and a numerically trivial but non-torsion divisor $D$ 
(here we need the assumption $k \not = \Fp$). 
Let $X = \mathrm{Proj}_C(\mcO_C(-P) \oplus \mcO_C(D))$, let $\pi \colon X \to C$ be the natural projection, and let $F = \pi^*(P)$.

By \cite[Ch.\ V. Corollary 2.11]{Hartshorne}, we have $-K_X \equiv 2C_0 + F$, where $C_0$ is the canonical section. Take $\Delta = C_0$. Then $-(K_X + \Delta) \equiv C_0 + F$ is a nef divisor. Further, it is big, as $(C_0 + F)^2 = 1$. However, $C_0+F$ cannot be semiample, because if it was, then the associated semiample fibration would contract $C_0$, but $(C_0 + F)|_{C_0} = (\pi^*D)|_{C_0}$ is non-torsion. Hence, $X$ is not a Mori dream space. 
\end{proof}

\subsection{Three-dimensional case}

In this subsection, we give a negative answer to Question \ref{question:main}: 
\begin{theorem} \label{theorem:main_mds} 
For any algebraically closed field $k$, there exists a log canonical projective three-dimensional pair $(X,\Delta)$ 
defined over $k$, such that $-(K_X+\Delta)$ is nef and big, 
but the nef cone $\mathrm{Nef}(X)$ is not a rational polyhedron. 
In particular, $X$ is not a Mori dream space.
\end{theorem}
The example is constructed by taking a blow-up of a projective cone of a K3 surface which is not a Mori dream space.
We would like to express our gratitude to Professors Yoshinori Gongyo and Keiji Oguiso for 
their suggesting the way of constructing the example.
\begin{proposition} \label{proposition:k3} 
For any algebraically closed field $k$, there exists a K3 surface $S$ defined over $k$ 
whose nef cone $\mathrm{Nef}(S)$ is not a rational polyhedron. In particular, $S$ is not a Mori dream space.
\end{proposition}
For any algebraically closed field $k$, it is known that there exists a K3 surface $S$ defined over $k$ 
which has an infinite automorphism group 
(see \cite[Corollary]{Ueno} and \cite[Theorem 1.1]{DK}, for instance). 
Therefore, Proposition \ref{proposition:k3} follows from the following proposition, 
whose proof was suggested to us by K.\ Oguiso. 
\begin{proposition}\label{proposition:k3_automorphisms} 
Let $S$ be a K3 surface defined over an algebraically closed field $k$. 
If the nef cone $\mathrm{Nef}(S)$ of $S$ is a rational polyhedron, then the automorphism group $\mathrm{Aut}(S)$ is finite.
\end{proposition}
\begin{proof}
Let $l_1, \ldots, l_{c}$ be primitive elements of $\mathrm{NS}(S)$ such that 
$\mbR_{\geq 0}l_i$ for $1 \leq i \leq c$ are the extremal rays of  $\mathrm{Nef}(S)$. 
As $\mathrm{Aut}(S)$ acts on $\mathrm{Nef}(S)$, it also acts on the set $\{ l_i \}_{1 \leq i \leq c}$. 
Let $G$ be the finite group of automorphisms of this set  $\{ l_i \}_{1 \leq i \leq c}$, 
which is isomorphic to the $c$-th symmetric group.

The above paragraph implies that there is a natural map $\mathrm{Aut}(S) \to G$. 
It is enough to show that its kernel $K$ is finite. 
To this end, note that $H^0(S,T_S) = 0$, and so $K \subseteq \mathrm{Aut}(S)$ is discrete. 
Furthermore, $K$ acts as the identity on $\Pic(S)$, due to the fact that $\mathrm{NS}(S)  = \Pic(S)$.

Take any very ample divisor $H$, and let $S \subseteq \mbP^N$ be the embedding induced by it. 
Since $K$ acts as the identity on $\Pic(S)$, the group $K$ preserves $H$, and so $K$ preserves 
the embedding $S \subseteq \mbP^N$. 
In particular, $K \subseteq \mathrm{PGL}_{N+1} (k)$. 
As $K$ is discrete and $\mathrm{PGL}_{N+1} (k)$ is Noetherian, the group $K$ must be finite.
\end{proof}

Now, we can proceed with the proof of the main theorem of this section.
\begin{proof}[Proof of Theorem \ref{theorem:main_mds}]
Let $S$ be a K3 surface whose nef cone $\mathrm{Nef}(S)$ is not a rational polyhedron (Proposition \ref{proposition:k3}). 
Choose an ample line bundle $A$ on $S$, and embed $S$ into a projective space $\mathbb{P}^N$ via some multiple of $A$, 
so that $S \subseteq \mathbb{P}^N$ is projectively normal.

Let $X \subseteq \mathbb{P}^{N+1}$ be the projective cone of $S \subseteq \mathbb{P}^N$, and let $\overline{X}$ be the blow-up of $X$ at the vertex of the cone. We have the following morphisms: 
\[\xymatrix{
\overline{X} \ar[d] _{p} \ar[r] ^{f} & X \\
S,  & \\ 
}\]
where $p \colon \overline{X} \to S$ is the natural $\mbP^1$-bundle on $S$. Note that $f$ is a resolution of singularities of $X$, and its exceptional divisor $E$ is isomorphic to $S$ by $p|_{E}$.

Set $\Delta \coloneq E$. Then, by \cite[Section 5.1]{Gongyo}, we have that $-(K_{\overline{X}} + \Delta)$ is nef and big. 
Further, note that $\Pic(\overline{X}) = \mathrm{NS}(\overline{X})$. Indeed, 
any numerically trivial line bundle $L$ on $\overline{X}$ descends to $S$ by the Grauert theorem 
(\cite[Ch.\ III. Corollary 12.9]{Hartshorne}), 
and any numerically trivial line bundle on a K3 surface is trivial. 

Let $F \simeq \mbP^1$ be a general fibre of $p$. Consider a subcone
\[
\{ D \in \mathrm{Nef}(\overline{X}) \mid D \cdot F = 0 \} \subseteq \mathrm{Nef}(\overline{X}).
\]
By the above paragraph, this subcone is isomorphic to $\mathrm{Nef}(S)$, and hence it is not a rational polyhedron. In particular, 
$\mathrm{Nef}(\overline{X})$ is not a rational polyhedron, and $\overline{X}$ is not a Mori dream space.
\end{proof}

\begin{remark}
Let $(X,\Delta)$ be a log canonical  projective three-dimensional pair defined over $\Fp$.
Let $D$ be a big and nef $\mbQ$-Cartier $\mbQ$-divisor on $X$. Let 
\[
\overline{\mathrm{NE}}(X)|_{<(K_X + \Delta + D)}
\]
be the cone of effective curves $C$ satisfying $C \cdot (K_X + \Delta + D)  < 0$. When $(X,\Delta)$ is klt, then $\overline{\mathrm{NE}}(X)|_{<(K_X + \Delta + D)}$ is a rational polyhedron by the cone theorem. Hence, it is natural to ask whether this still holds without the klt assumption. Unfortunately, Theorem \ref{theorem:main_mds} shows that this is not the case, and so the cone theorem does not extend to a setting similar to the one of Theorem \ref{theorem:main}.
\end{remark}
  
\section*{Acknowledgements}
We would like to thank Paolo Cascini, Yoshinori Gongyo, Yujiro Kawamata, Diletta Martinelli, Keiji Oguiso, 
Hiromu Tanaka and Joe Waldron for useful comments and suggestions.
The first author is supported by the Grant-in-Aid for Scientific Research (KAKENHI No. 25-3003). The second author is funded by EPSRC.

\end{document}